\documentclass[a4paper,12pt]{article}
\usepackage[english]{babel}
\usepackage[T1]{fontenc}
\usepackage{amsthm,enumerate,amsmath,amscd,amssymb}
\usepackage{graphicx}

\newcommand{\R}{\mathbb{R}}
\renewcommand{\H}{{\mathbb{H}}}
\newcommand{\B}{{\mathbb{B}}}
\newcommand{\Rn}{{\mathbb{R}^n}}
\newcommand{\Hn}{{\mathbb{H}^n}}
\newcommand{\Bn}{{\mathbb{B}^n}}
\newcommand{\arcsinh}{\textnormal{arcsinh}\,}

\newcommand{\arctanh}{\textnormal{arctanh}\,}

\newcommand{\PP}{ {\R^2 \setminus \{ 0 \}} }
\newcommand{\PS}{ {\Rn \setminus \{ 0 \}} }
\newcommand{\Ball}[3]{B_{#1}(#2,#3)}
\newcommand{\Sphere}[3]{S_{#1}(#2,#3)}
\newcommand{\comment}[1]{}

\newcounter{minutes}
\divide\time by 60
\newcounter{hours}
\multiply\time by 60
\addtocounter{minutes}{-\time}

\swapnumbers
\numberwithin{equation}{section}
\theoremstyle{plain}

\newtheorem{theorem}[equation]{Theorem}
\newtheorem{corollary}[equation]{Corollary}
\newtheorem{proposition}[equation]{Proposition}
\newtheorem{lemma}[equation]{Lemma}

\newtheorem{remark}[equation]{Remark}

\begin{document}

\begin{center}
{\bf \large Inclusion relations of hyperbolic type metric balls II}
\end{center}

\vspace{1mm}

\begin{center}
{\large Riku Klén and Matti Vuorinen}
\end{center}

\vspace{1mm}

\begin{abstract}
  Inclusion relations of metric balls defined by the hyperbolic, the quasihyperbolic,
  the $j$-metric and the chordal metric will be studied. The hyperbolic metric,
  the quasihyperbolic metric and the $j$-metric are considered in the unit ball.
\end{abstract}



\def\thefootnote{}
\footnotetext{
Department of Mathematics and Statistics, University of Turku FI-20014, FINLAND, riku.klen@utu.fi, vuorinen@utu.fi, phone +358 2 333 6675, fax
+358 2 231 0311 \\
\texttt{\tiny File:~\jobname .tex,
          printed: \number\year-\number\month-\number\day,
          \thehours.\ifnum\theminutes<10{0}\fi\theminutes}
}

\section{Introduction}

The most important metrics in the classical complex analysis are the
Euclidean and the hyperbolic metric. Studying quasiconformal mappings in $\Rn$, F.W. Gehring and
B.P. Palka \cite{gp} introduced the quasihyperbolic metric, which
plays the role of the hyperbolic metric in the higher dimensions. The quasihyperbolic metric has recently been studied in \cite{ccq,rt}. There are also other hyperbolic type metrics like the distance ratio metric and the Apollonian metric, which have lately been studied
by various authors \cite{himps,hpws}.

Suppose that $(X,d_j), j=1,2,\,$ are two metric spaces with $X \subset {\mathbb R}^n$ such that both metrics
determine the same Euclidean topology. In order to understand the geometric structure of the spaces, it is
a fundamental question to study the corresponding balls and inclusion relations among balls with the same
center with respect to both of these two  metrics. In several classical cases such relations are well-known, comparing e.g. the
Euclidean balls and hyperbolic balls (cf. \cite[Section 2]{v2}). But this comparison problem makes sense in numerous non-classical cases
as well, as pointed out in \cite{vu3}. Very recently, such non-classical inclusion problems have been
studied in \cite{kv,m} and our goal here is to investigate the inclusion problem for hyperbolic type metrics in the unit ball.

In a metric space $(X,m)$ we define \emph{metric ball} or \emph{$m$-ball} with center $x \in X$ and radius $r>0$ by
\begin{equation}\label{mball}
  \Ball{m}{x}{r} = \{ y \in X \colon m(x,y) < r \}
\end{equation}
and \emph{metric sphere} with center $x \in X$ and radius $r>0$ by
\begin{equation*}
  \Sphere{m}{x}{r} = \{ y \in X \colon m(x,y) = r \}.
\end{equation*}
For Euclidean balls and spheres we use notation $B^n(x,r)$ and $S^{n-1}(x,r)$. For $x,y \in \Rn$ we use notation $l=[x,y]$ for the line segment joining $x$ to $y$, similarly $[x,y) = l \setminus \{ y \}$ and $(x,y]= l \setminus \{ x \}$. For $x \in G \subsetneq \Rn$ we denote by $d(x)$ the Euclidean distance between $x$ and $\partial G$.

The $n$-dimensional unit ball will be denoted by $\Bn$ and
half-space by $\Hn = \{ x \in \Rn \colon x_2 > 0 \}$. The
\emph{hyperbolic length} of a rectifiable curve $\gamma \subset \Bn$ is defined by
\[
  \ell_{\rho_\Bn}(\gamma) = \int_{\gamma}\frac{2|dz|}{1-|z|^2}
\]
and $\gamma \subset \Hn$ by
\[
  \ell_{\rho_\Hn}(\gamma) = \int_{\gamma}\frac{|dz|}{z_n}.
\]
The \emph{hyperbolic metric} in $G \in \{ \Bn,\Hn \}$ is
\[
  \rho_G(x,y) = \inf_\gamma \ell_{\rho_G}(\gamma),
\]
where the infimum is taken over all rectifiable curves in $G$ joining $x$ and $y$.

For a domain $G \subsetneq \Rn$, $n \ge 2$ the \emph{quasihyperbolic
length} of a rectifiable arc $\gamma \subset G$ is given by
\[
  \ell_k(\gamma) = \int_{\gamma}\frac{|dz|}{d(z)}
\]
and the \emph{quasihyperbolic metric} by
\begin{equation}\label{qhmetric}
  k_G(x,y) = \inf_\gamma \ell_k(\gamma),
\end{equation}
where the infimum is taken over all rectifiable curves in $G$ joining $x$ and $y$. Note that $k_\Hn = \rho_\Hn$.

The \emph{distance ratio metric} or \emph{$j$-metric} in a proper subdomain $G$ of the Euclidean space $\Rn$, $n \ge 2$, is defined by
$$
  j_G(x,y) = \log \left( 1+\frac{|x-y|}{\min \{ d(x),d(y) \}}\right).
$$
The distance ratio metric satisfies the triangle inequality by \cite[Lemma 2.2]{s}. If the domain $G$ is understood from the context we use the notation $j$ instead of $j_G$ and $k$ instead of $k_G$. The distance ratio metric was first introduced by F.W. Gehring and B.G. Osgood \cite{go}, and in the above form by M. Vuorinen \cite{v2}. The metric space $(G,j_G)$ is not geodesic for any domain $G$ \cite[Theorem 2.10]{k1}.

The \emph{chordal metric} in $\overline{\Rn} = \Rn \cup \{ \infty \}$ is defined by
\[
  q(x,y) = \left\{ \begin{array}{ll}
    \displaystyle \frac{|x-y|}{\sqrt{1+|x|^2}\sqrt{1+|y|^2}}, & x \neq \infty \neq y,\\
    \displaystyle \frac{1}{\sqrt{1+|x|^2}}, & y=\infty.
  \end{array} \right.
\]
The metric space $(\overline{\Rn},q)$ is not geodesic

We find radii $m(r)$ and $M(r)$ for $d_1,d_2 \in \{ q,k_G,j_G \}$ such that
\[
  B_{d_1}(x,m(r)) \subset B_{d_2}(x,r) \subset B_{d_1}(x,M(r))
\]
for all $x \in G$. This kind of estimates can be used to compare
metrics, because $B_{m_2}(x,r) \subset B_{m_1}(x,r)$ for all $r$ implies
$m_1(x,y) \le m_2(x,y) < r$ for all $y \in B_{m_2}(x,r)\,.$

The geometry of the $j$-metric balls is easily described in $\Hn$, $\PS$ and polygons in the case $n=2$, as we will point out in Lemma \ref{curvature}, Remark \ref{remcurv} and Figure \ref{jballs}. However, when the boundary of the domain does not consist of lines and isolated points the situation becomes more complicated. Already in the unit ball the geometry of the $j$-metric balls differ significantly from the other cases.

The following theorem is our main result.

\begin{theorem}\label{mainthm1}
  Let $G = \Bn$, $x \in G$ and $r > 0$. Then
  \[
    B_j(x,m_1(r)) \subset B_\rho(x,r),
  \]
  \[
    B_j(x,m_2(r)) \subset B_k(x,r),
  \]
  \[
    B_j(x,m_3(r)) \subset B_q(x,r), \quad r < (1-|x|)/\sqrt{2(1+|x|^2)},
  \]
  where
  \begin{eqnarray*}
    m_1(r) & = & \log \left( 1+2 \sinh \frac{r}{2} \right),\\
    m_2(r) & = & \log \left( 1+2 \sinh \frac{r}{4} \right),\\
    m_3(r) & = & \log \left( 1+\frac{r}{\sqrt{1-r^2}} \right).
  \end{eqnarray*}
\end{theorem}

\begin{remark}
We show that
  \[
    \frac{r}{2} < m_1(r) < r, \quad \frac{r}{4} < m_2(r) < \frac{r}{2}, \quad \frac{4r}{5} < m_3(r),
  \]
  where $m_1(r)$, $m_2(r)$ and $m_3(r)$ are as in Theorem \ref{mainthm1}.

  By a simple computation we obtain
  \[
    m_1'(r) = \frac{\cosh \frac{r}{2}}{1+2\sinh \frac{r}{2}} < 1,
  \]
  where the inequality follows from the fact $\cosh(r/2) <1+2 \sinh (r/2)$. Since $m_1(0)=0$ we obtain $m_1(r) < r$.

  The lower bound for $m_1(r)$ follows from
  \[
    m_1(r) = \log (1+e^{r/2}-e^{r/2}) > \log (e^{r/2}) = \frac{r}{2}.
  \]

  The upper and lower bounds for $m_2(r)$ follow from the bounds of $m_1(r)$, because $m_2(r) = m_1(r/2)$.

  Since
  \[
    m_3''(r) = \frac{1-3 r (r+\sqrt{1-r^2})}{(1-r^2)^2(r+\sqrt{1-r^2)^2}}
  \]
  we know that $m_3'(r)$ attains its minimum on the interval $(0,1)$ at $r_0 = ( (5-\sqrt{17})/3  )^{1/2} /2$ and
  \[
    m_3'(r_0) = \frac{24 \sqrt{3}}{(7+\sqrt{17}) (\sqrt{5-\sqrt{17}}+\sqrt{7+\sqrt{17}})} > \frac{4}{5}.
  \]
 Thus $4r/5 < m_3$.
\end{remark}

\section{Preliminary results}

In this section we introduce preliminary results such as properties
of hyperbolic type metric balls and relations between hyperbolic
type metrics.

The \emph{curvature} of a plane curve parameterized in polar coordinates is defined by
\[
  \kappa = \frac{r(\theta)^2+2r'(\theta)^2-r(\theta)r''(\theta)}{(r(\theta)^2+r'(\theta)^2)^{3/2}}.
\]
Note that if a curve has a constant curvature, then it is a circular arc.

The following lemma describes the geometric shape of the $j$-metric balls. Some examples of $j$-metric disks are shown in Figure \ref{jballs}.

\begin{lemma}\label{curvature}
  Let $G \subset \R^2$, $x \in G$ and $r > 0$. Then curvature of $S_j(x,r)$
  \begin{itemize}
    \item[(1)] in $G = \R^2 \setminus \{ 0 \}$ is
    \[
      \kappa = \left\{ \begin{array}{ll} \displaystyle \frac{1}{|x|(e^r-1)}, & \textrm{in } \R^2 \setminus B^2(0,x),\\ \displaystyle \frac{e^r |2-e^r|}{|x|(e^r-1)}, & \textrm{in } B^2(0,x). \end{array} \right.
    \]
    \item[(2)] in $G = \H^2$ with $x_1 = 0$ is
    \[
      \kappa = \left\{ \begin{array}{ll} \displaystyle \frac{1}{(e^r-1)x_2}, & \textrm{in } \{ y \in \H^2 \colon y_2 > x_2 \},\\ \displaystyle \frac{|x|^2}{(e^r-1)(|x|^2+t^2)}, & \textrm{in } \{ y \in \H^2 \colon y_2 < x_2 \}, \end{array} \right.
    \]
    where $|t| < |x|(e^r-1)$.
    \item[(3)] in $G = \B^2$ is
    \[
      \kappa = \left\{ \begin{array}{ll} \displaystyle \frac{1}{(e^r-1)(1-|x|)}, & \textrm{in } B^2(0,|x|),\\ \displaystyle \frac{f(\alpha)^2+2f'(\alpha)-f(\alpha)f''(\alpha)}{(f(\alpha)^2+f'(\alpha)^2)^{3/2}}, & \textrm{in } \B^2 \setminus \overline{B^2(0,|x|)}, \end{array} \right.
    \]
    where $\alpha$ is the angle $\measuredangle(x,0,y)$ for $y \in S_j(x,r)$ and
    \[
      f(\alpha) = \frac{(1-e^r)^2-\beta -\sqrt{1+(e^{2r}-2e^r)(1+|x|^2)-2(e^r-1)^2 \beta+\beta^2}}{e^r(e^r-2)}
    \]

    for $\beta = |x| \cos \alpha$ and $\alpha \in [0,\gamma]$, where
    \[
      \gamma = \left\{ \begin{array}{ll} \pi, & r \ge \log ((1+|x|)/(1-|x|)),\\ \displaystyle 2 \arcsin \frac{(e^r-1)(1-|x|)}{2|x|} & r < \log ((1+|x|)/(1-|x|)). \end{array} \right.
    \]

    \item[(4)] in $\{ y \in G \colon d(y) \ge d(x) \}$ for any domain $G$ is
    \[
      \kappa = \frac{1}{(e^r-1)d(x)}.
    \]
  \end{itemize}
\end{lemma}

\begin{proof}
  Let $G \subset \R^2$, $x \in G$ and $r > 0$.
  \begin{itemize}
    \item[(1)] By definition of the $j$-metric the $j$-sphere consists of two circular arcs, or in the case $r=\log 2$, a circular arc and a line segment. The assertion follows from proof of \cite[Theorem 3.1]{k1}.

    \item[(2)] The case $S_j(x,r) \cap \{ y \in \H^2 \colon y_2 > x_2 \}$ is similar to (2). Therefore, we consider $S = S_j(x,r) \cap \{ y \in \H^2 \colon y_2 < x_2 \}$. By the definition of the $j$-metric we obtain that $S = \{ y \in \H^2 \colon y=(t,f(t)), \, t \in (-|x|(e^r-1),|x|(e^r-1) ) \}$ for the function
    \[
      f(t) = \left| \frac{\sqrt{|x|^2+e^r(e^r-2)(|x|^2+t^2)}-|x|}{e^r(e^r-2)} \right|.
    \]
    By a straightforward computation we obtain that the curvature of $S$ at point $(t,f(t))$, $t \in (-|x|(e^r-1),|x|(e^r-1))$, is
    \[
      \kappa(t) = \frac{|f''(t)|}{(1+f'(t)^2)^{3/2}} = \frac{|x|^2}{(e^r-1)(|x|^2+t^2)^{3/2}}.
    \]

    \item[(3)] To simplify notation we may assume $x = (x_1,0)$ for $x_1 \in [0,1)$. We divide $S_j(x,r)$ into two cases $S_1 = S_j(x,r) \cap B^2(0,|x|)$ and $S_2 = S_j(x,r) \cap (\B^2 \setminus  \overline{B^2(0,|x|)} ) )$. The set $S_2$ is always nonempty, whereas the set $S_1 = \emptyset$ whenever $r > \log((1+|x|)/(1-|x|))$.

    Let us first consider $S_1$. By the definition of the $j$-metric for all $y \in S_1$ we have $|x-y| = (e^r-1)(1+|x|^2)$, and thus
    \[
      \kappa(t) = \frac{1}{(e^r-1)(1-|x|)}.
    \]

    Let us finally consider $S_2$. The assertion follows by the definition of curvature, if the function $f(\alpha) = |y|$, for the point $y \in S_2$ and $\measuredangle(x,0,y) = \alpha$. By the definition of the $j$-metric for $y \in S_2$ we obtain $|x-y| = (e^r-1)(1-|y|)$ and by the law of cosines we obtain
    \[
      (e^r-1)^2(1-|y|)^2 = |x|^2+|y|^2-2 |x| |y| \cos \alpha,
    \]
    which is equivalent to $|y| = f(\alpha)$. The sign in $f(\alpha)$ was chosen to be minus, because otherwise the values would have been greater than or equal to 1.

    \item[(4)] The assertion follows from the definition of the $j$-metric as in the case (3).
  \end{itemize}
\end{proof}

\begin{remark}\label{remcurv}
  (1) By Lemma \ref{curvature} the boundary $\partial \Ball{j}{x}{r}$ consists of two circular arcs in the case $G=\PP$ and a circular arc and a part of a conic section (hyperbola if $r > \log 2$, parabola if $r=\log 2$ or ellipse if $r < \log 2$) in the case $G = \H^2$, see Figure \ref{jballs}. In the case $G = \B^2$ the boundary $\partial \Ball{j}{x}{r}$ is more complicated as it may not even contain circular arc.

  (2) By Lemma \ref{curvature} (2) the boundary $\partial \Ball{j}{x}{r}$ can be
  formed in a polygonal domain $P \subset {\mathbb R}^2$. First the medial axis of $P$ needs to be formed. In a convex polygon the medial axis consists of line segments and can be found as Voronoi diagram \cite{obs}. If the polygon is not convex, then the medial axis can contain parts of conic sections. However, the medial axis is unique and it divides $P$ into smaller domains $A_i$. In each $A_i$ the boundary $\partial \Ball{j}{x}{r}$ consists of circular arcs, when $A_i \cap (\partial \Ball{j}{x}{r}) \subset \{ z \in P \colon d(z) \ge d(x) \}$, and parts of a conic section similarly as above in the case $G=\H^2$, see Figure \ref{jballs}.
\end{remark}

\begin{figure}[!ht]
  \begin{center}
    \includegraphics[width=28mm]{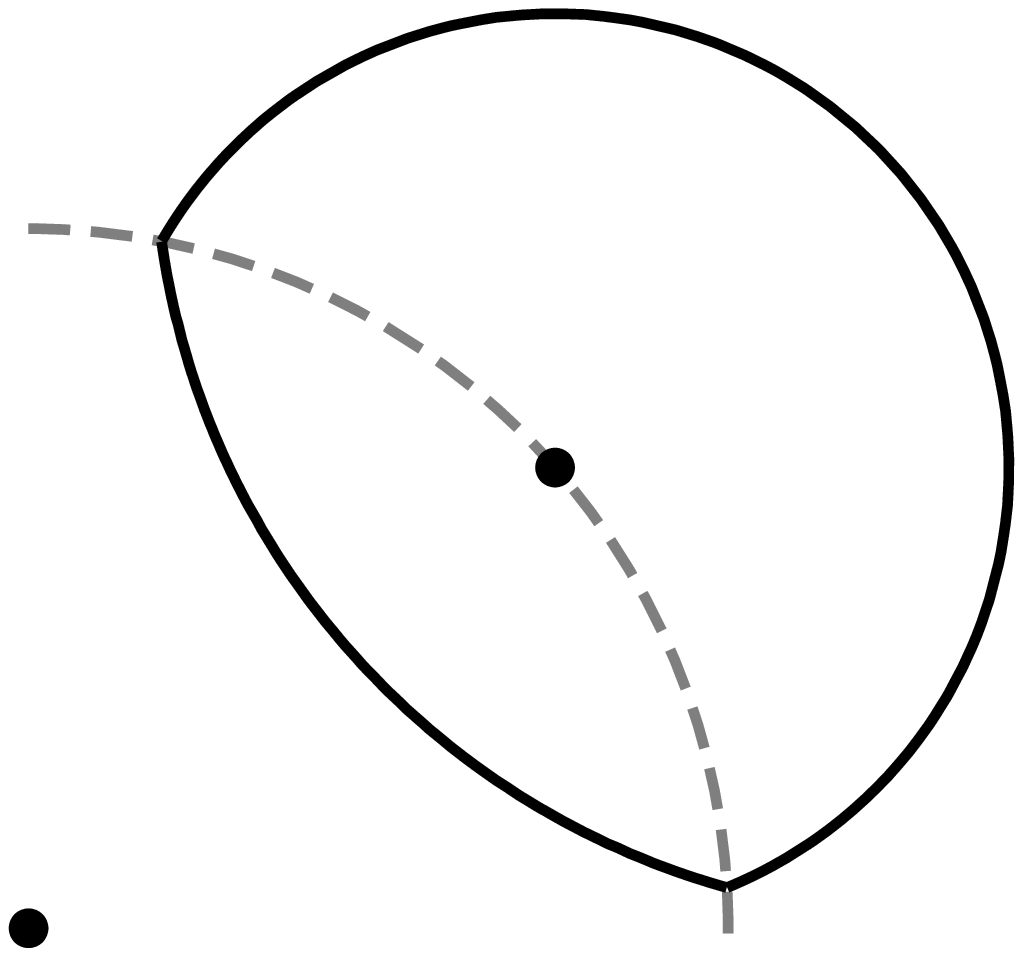}\hspace{4mm}
    \includegraphics[width=33mm]{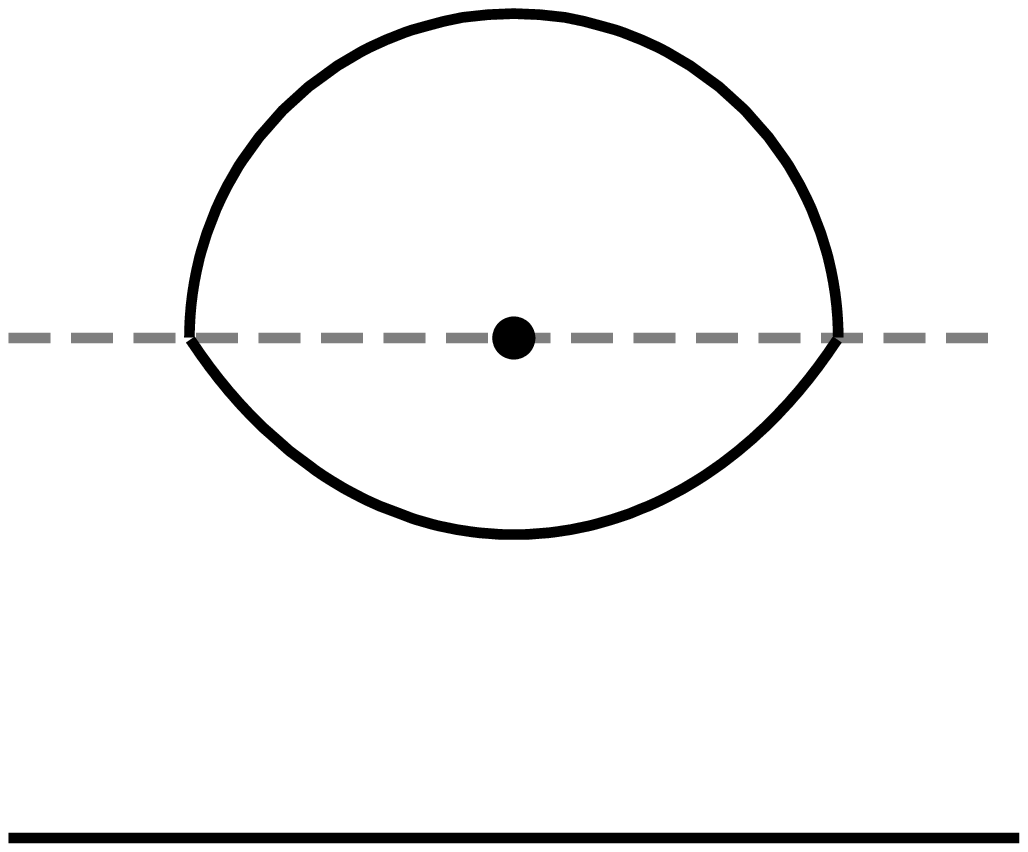}\hspace{4mm}
    \includegraphics[width=52mm]{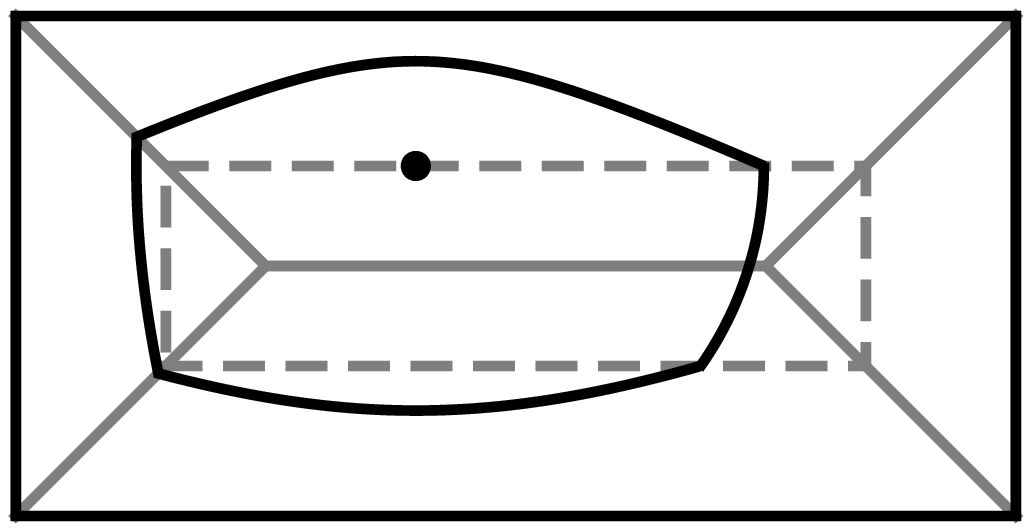}
    \caption{Examples of $j$-disks $\Ball{j}{x}{r}$ in punctured plane (left), half-plane (middle) and in a rectangle (right). Gray line is the medial axis and dashed gray line is the set $\{ z \in G \colon d(z) = d(x) \}$.\label{jballs}}
  \end{center}
\end{figure}

For the sake of easy reference we recapitulate a few basic facts in
the next result. For part (1) and (4), see e.g. \cite[Section
7]{avv}, for (2) and (3) \cite[Section 2]{v2}.

\begin{proposition}\label{aputulos}
  For all $x,y \in \Bn$
\begin{enumerate}
  \item[(1)]
  \[
    \rho_\Bn(x,y) \le 2 k_\Bn(x,y) \le 2\rho_\Bn(x,y),
  \]
  \item[(2)]
  \[
    \rho_\Bn(x,y) = 2 \arcsinh \frac{|x-y|}{\sqrt{1-|x|^2}\sqrt{1-|y|^2}},
  \]
  \item[(3)]and for $r>0$
  \[
    B_{\rho_\Bn}(x,r) = B^n \left( \frac{x(1-t^2)}{1-|x|^2t^2} , \frac{(1-|x|^2)t}{1-|x|^2t^2} \right),
  \]
  where $t=\tanh (r/2)$,
  \item[(4)] and for $r \in (0,1/\sqrt{1+|x|^2})$
  \[
    \Ball{q}{x}{r} = B^n \left( \frac{x}{1-r^2(1+|x|^2)} , \frac{r(1+|x|^2)\sqrt{1-r^2}}{1-r^2 (1+|x|^2)} \right).
  \]
\end{enumerate}
\end{proposition}
Note that in (4) we have $\Ball{q}{x}{r} \subset \Bn$, if $r \in (0,(1-|x|)/\sqrt{2(1+|x|^2)})$.

\begin{lemma}\label{BjBn}
  Let $G=\Bn$, $x = t e_1$, $t \in [0,1)$, $r>0$ and
  \[
    \partial \Ball{j}{x}{r} \cap \{ z \in \Rn \colon z=s e_1,\, s\in\R \} = \{ y_1,y_2 \}
  \]
  with $|y_2| \le |y_1|$. Then
  \[
    B^n(x,|x-y_1|) \subset \Ball{j}{x}{r} \subset B^n(x,|x-y_2|).
  \]
\end{lemma}
\begin{proof}
  By the selection of $y_1$ and $y_2$ we have $j(x,y_1) = j(x,y_2) = r$.

  We show that $B^n(x,|x-y_1|) \subset \Ball{j}{x}{r}$. Let $y \in B^n(x,|x-y_1|)$. Because $d(y_1) < \min \{ d(x),d(y) \}$ we have
  \begin{eqnarray*}
    j(x,y) & = & \log \left( 1+\frac{|x-y|}{\min \{ d(x),d(y) \}} \right) < \log \left( 1+\frac{|x-y_1|}{\min \{ d(x),d(y) \} } \right)\\
    & \le & \log \left( 1+\frac{|x-y_1|}{d(y_1)} \right) = j(x,y_1) = r.
  \end{eqnarray*}

  We show that $\Ball{j}{x}{r} \subset B^n(x,|x-y_2|)$. Let $z \in \Ball{j}{x}{r}$. We divide the proof into two cases: $y_2 \in [-x,x)$ and $y_2 \in (-x/|x|,-x)$.

  If $y_2 \in [-x,x)$, then $d(x) \le d(y_2)$ and thus $j(x,z) < j(x,y_2)$ is equivalent to
  \[
    \frac{|x-z|}{ \min \{ d(x),d(z) \} } < \frac{|x-y_2|}{d(x)}
  \]
  implying $|x-z| < |x-y_2|$.

  If $y_2 \in (-x/|x|,-x)$, then $d(y_2) < d(x)$. Inequality $j(x,z) < j(x,y_2)$ is equivalent to
  \[
    \frac{|x-z|}{ \min \{ d(x),d(z) \} } < \frac{|x-y_2|}{d(y_2)}
  \]
  implying $|x-z| < |x-y_2|$, if additionally $d(z) \le d(y_2)$. If $d(z) > d(y_2)$, then immediately $|x-z|<|x-y|$.

  In both cases we obtain that $|x-z| > |x-y_2|$ and thus the assertion follows.
\end{proof}

\begin{lemma}\label{jlemma}
  Let $G \in \{ \PS, \Hn, \Bn \}$, $r > 0$ and $x \in G$. Then
  \[
    B = B^n \left( \frac{y+z}{2},\frac{|y-z|}{2} \right) \subset B_j(x,r),
  \]
  where $y,z \in l \cap \partial B_j(x,r)$ with $d(y) \le d(z)$ and $l$ is the line that contains $x$ and a boundary point of $G$ that is closest to $x$. Moreover, $B$ is the largest Euclidean ball contained in $B_j(x,r)$.
\end{lemma}
\begin{proof}
  The cases $G \in \{ \PS,\Bn \}$ follow easily from Lemma \ref{curvature} (1) and (2).

  Let us consider $G = \Bn$. Now $y=x(1-e^{-r}(1-|x|)/|x|$ and
  \[
    z = \left\{ \begin{array}{ll} \displaystyle x \frac{1-e^r(1-|x|)}{|x|}, & \textrm{if }\displaystyle r \le \log \frac{1}{1-|x|},\\ \displaystyle x \frac{e^r(1-|x|)-1}{|x|}, & \textrm{if }\displaystyle r > \log \frac{1}{1-|x|} \textrm{ and } \log \frac{1+|x|}{1-|x|} \ge r,\\ \displaystyle x \frac{1-e^{-r}(1+|x|)}{|x|}, & \textrm{if }\displaystyle r > \log \frac{1}{1-|x|} \textrm{ and } \log \frac{1+|x|}{1-|x|} < r. \end{array}  \right.
  \]
  Thus
  \[
    \frac{y+z}{2} = \left\{ \begin{array}{ll} \displaystyle \frac{x(1-(1-|x|)\cosh r)}{|x|}, & r \le \displaystyle \log \frac{1+|x|}{1-|x|},\\ \displaystyle -\frac{x e^{-r}}{|x|}, & \displaystyle r > \log \frac{1+|x|}{1-|x|} \end{array} \right.
  \]
  and
  \[
    \frac{|y-z|}{2} = \left\{ \begin{array}{ll} \displaystyle (1-|x|)\sinh r, & r \le \displaystyle \log \frac{1+|x|}{1-|x|},\\ \displaystyle 1-e^{-r}, & \displaystyle r > \log \frac{1+|x|}{1-|x|}. \end{array} \right.
  \]
\end{proof}

\begin{remark}\label{jremark}
  Lemma \ref{jlemma} is not true in general. For $G = \R^2 \setminus \{ -e_1,e_1 \}$ does not hold for $x = e_2$ and $r=\log 2$. In this case $\partial \Ball{j}{x}{r}$ consists of two perpendicular line segments and a circular arc. The line segments are $[0,a(e_1+e_2)]$ and $[0,a(-e_1+e_2)]$, where $a = (1+\sqrt{3})/2$. See Figure \ref{jdiskremark}.

  However, the following question is open: Is Lemma \ref{jlemma} true in convex domains?
\end{remark}

\begin{figure}[!ht]
  \begin{center}
    \includegraphics[width=40mm]{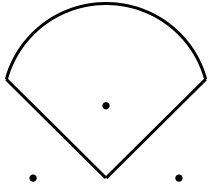}
    \caption{Boundary $\partial \Ball{j}{e_2}{\log 2}$ in $G = \R^2 \setminus \{ -e_1,e_1 \}$.\label{jdiskremark}}
  \end{center}
\end{figure}

\begin{lemma}\label{technicallemma}
  For $a,b \in [0,1]$ we have
  \begin{itemize}
    \item[(1)] $\min \{ 1-a,1-b \}(1+\max \{ a,b \}) \le \sqrt{1-a^2} \sqrt{1-b^2}$.
  \end{itemize}
  For $x,y \in \Bn$ and $r \ge \arcsinh (2|x|/(1-|x|^2))$ we have
  \begin{itemize}
    \item[(2)] $\min \{ d(x),d(y) \} (1+|x|) \le \sqrt{1-|x|^2}\sqrt{1-|y|^2}$,
    \item[(3)] $\displaystyle \frac{2|x|}{1-|x|}-\frac{1+|x|}{1/\tanh(r/2)-|x|} \le (1+|x|)\frac{e^r-1}{2}$.
  \end{itemize}
\end{lemma}

\begin{proof}
  We consider first (1). We easily obtain
  \[
    \min \{ 1-a,1-b \}(1+\max \{ a,b \}) = 1-\max \{ a,b \}^2 \le \sqrt{1-a^2} \sqrt{1-b^2}.
  \]

  Part (2) follows from (1).

  Let us then consider (3), which is equivalent to showing that the function
  \[
    f(r) = (1+|x|)\frac{e^r-1}{2}-\frac{2|x|}{1-|x|}+\frac{1+|x|}{1/\tanh(r/2)-|x|}
  \]
  is nonnegative. Since
  \[
    f'(r) = \frac{1+|x|}{2} \left( e^r+\frac{1}{( \cosh (r/2)-|x| \sinh(r/2) )^2} \right) > 0
  \]
  and $f(\arcsinh (2|x|/(1-|x|^2))) = |x|^2/(1-|x|) \ge 0$ the assertion follows.
\end{proof}

\section{Inclusion relations of metric balls}

In this section we consider metric balls in unit ball $G = \Bn$. Since we do not know the exact form of the quasihyperbolic ball we need to use the hyperbolic balls.

\begin{theorem}\label{jrhoj}
  Let $G = \Bn$, $x \in G$ and $r > 0$. Then
  \[
    B_j(x,m) \subset B_\rho(x,r) \subset B_j(x,M),
  \]
  where $m = \max \{ m_1, m_2 \}$
  \[
    m_1 =  \log \left( 1+(1+|x|)\sinh \frac{r}{2} \right) , \quad m_2 = \log \left( 1+(1-|x|)\frac{e^r-1}{2} \right)
  \]
  and
  \[
    M = \log \left( 1+(1+|x|)\frac{e^r-1}{2} \right).
  \]
  Moreover, the inclusions are sharp and $M/m \to 1$ as $r \to 0$.
\end{theorem}
\begin{proof}
  We prove inclusion $B_j(x,m) \subset B_\rho(x,r)$. Let us first assume $y \in B_j(x,m_1)$, which is equivalent to
  \begin{equation}\label{absxyestimate1}
    |x-y| < \min \{ d(x),d(y) \} (1+|x|)\sinh (r/2).
  \end{equation}
  Since $\sinh$ and $\arcsinh$ are increasing we obtain by Proposition \ref{aputulos} (2) and \eqref{absxyestimate1}
  \begin{eqnarray*}
    \rho(x,y) & \le & 2 \arcsinh \displaystyle \frac{\min \{ d(x),d(y) \} (1+|x|)\sinh (r/2)}{\sqrt{1-|x|^2}\sqrt{1-|y|^2}}\\
    & \le & 2 \arcsinh ( \sinh (r/2) ) \le r,
  \end{eqnarray*}
  where the second inequality follows from Lemma \ref{technicallemma} (2). Now $y \in B_\rho(x,r)$ and thus $B_j(x,m_1) \subset B_\rho(x,r)$.

  Let us then assume $y \in \partial B_j(x,m_2)$ and $m_1 < m_2$. Since $m_1 < m_2$ is equivalent to $r > 4 \arctanh |x|$, we obtain by Lemma \ref{aputulos} (2) that $S^{n-1}(0,|x|) \subset \Ball{\rho}{x}{r}$. Thus $|x| < |y|$, and $j(x,y) = m_2$ is equivalent to
  \[
    \frac{|x-y|}{1-|y|} = (1-|x|)\frac{e^r-1}{2}.
  \]
  Now
  \begin{eqnarray*}
    \rho(x,y) & = & 2 \arcsinh \left(  \displaystyle \frac{|x-y|}{(1-|y|)\sqrt{1-|x|^2}} \sqrt{\frac{1-|y|}{1+|y|}} \right)\\
    & = & 2 \arcsinh \left(  \displaystyle \frac{(1-|x|)(e^r-1)}{2\sqrt{1-|x|^2}} \sqrt{\frac{1-|y|}{1+|y|}} \right)
  \end{eqnarray*}
  and $\rho(x,y) \le \rho(x,z_1)$ for $z_1 \in \partial B_j(x,m_2)$ with $|z_1| \le |y|$. In other words $z_1 = l \cap \partial \Ball{j}{x}{r}$, where $l = \{ u \in \Bn \colon u=sx,\,s<|x| \}$, and
  \[
    |z_1| = 1-\frac{2(1+|x|)}{1+|x|+e^r(1-|x|)}.
  \]
  Thus
  \[
    \rho(x,y) \le \rho(x,z_1) = 2 \arcsinh \frac{|x|+|z_1|}{\sqrt{1-|x|^2}\sqrt{1-|z_1|^2}}
  \]
  and $y \in \Ball{\rho}{x}{r}$ implying $B_j(x,m_2) \subset B_\rho(x,r)$.

  We show that $m$ is sharp. If $S^{n-1}(0,|x|) \cap ( \partial \Ball{j}{x}{r} ) = \emptyset$, then $j(x,z_1) = m_2 = m$. Otherwise we can choose $z \in S^{n-1}(0,|x|) \cap ( \partial \Ball{j}{x}{r} )$ and we obtain $j(x,z) = m_1 = m$.

  We prove next the inclusion $B_\rho(x,r) \subset B_j(x,M)$. We assume first that $y \in B_\rho(x,r)$ and $d(x) \le d(y)$, which is equivalent to $|y| \le |x|$. By Lemma \ref{curvature} (3) and Proposition \ref{aputulos} (3) $y \in B_j(x,M)$, if $j(x,z_1) \le M$ for $z_1 = l \cap \partial B_\rho(x,r)$, where $l = \{ u \in \Bn \colon u=sx,\,s<|x| \}$. If $|x-z_1| = |x|-|z_1|$, then by Proposition \ref{aputulos} (2) we have $r \le \arcsinh (2|x|/(1-|x|^2))$,
  \[
    |z_1| = \frac{2|x|+(|x|^2-1)\sinh r}{1+|x|^2-(|x|^2-1)\cosh r}
  \]
  and
  \[
    \frac{|x-z_1|}{\min \{ d(x),d(z_1) \}} = \frac{|x|-|z_1|}{1-|x|} = \frac{1+|x|}{1/\tanh(r/2)-|x|} \le (1+|x|)\frac{e^r-1}{2}
  \]
  implying $j(x,z_1) \le M$.
  If $|x-z_1| = |x|+|z_1|$, then by Proposition \ref{aputulos} (2) we have $r \ge \arcsinh (2|x|/(1-|x|^2))$,
  \[
    |z_1| = \frac{2|x|+(|x|^2-1)\sinh r}{-1-|x|^2+(|x|^2-1)\cosh r}
  \]
  and by Lemma \ref{technicallemma} (3)
  \[
    \frac{|x-z_1|}{\min \{ d(x),d(z_1) \}} = \frac{|x|+|z_1|}{1-|x|} = \frac{2|x|}{1-|x|}-\frac{1+|x|}{1/\tanh(r/2)-|x|} \le (1+|x|)\frac{e^r-1}{2}
  \]
  implying $j(x,z_1) \le M$.

  We assume then that $y \in \partial B_\rho(x,r)$ and $d(y) \le d(x)$. Now $y \in \partial B_\rho(x,r)$ is equivalent to
  \[
    \frac{|x-y|}{\sqrt{1-|y|}} = \sqrt{1+|y|}\sqrt{1-|x|^2}\sinh \frac{r}{2}
  \]
  and thus by Lemma \ref{aputulos} (3)
  \[
    j(x,y) = \log \left( 1+\frac{\sqrt{1+|y|}}{\sqrt{1-|y|}}\sqrt{1-|x|^2}\sinh \frac{r}{2} \right) \le j(x,z_2)
  \]
  for $z_2 = l \cap \partial B_\rho(x,r)$, where $l = \{ u \in \Bn \colon u=sx,\,s>|x| \}$. By Proposition \ref{aputulos} (2) we obtain
  \[
    |z_2| = 1- \frac{2(1-|x|)}{1-|x|+e^r(1+|x|)}
  \]
  and
  \[
    j(x,z_2) = \log \left( 1+\frac{|z_2|-|x|}{1-|z_2|} \right) = M
  \]
  implying the claim. This also shows that $M$ is sharp.

  By the l'H\^opital rule we obtain
  \[
    \lim_{r \to 0} \frac{M}{m} = \lim_{r \to 0} \frac{(1+(e^r-1)(1+|x|)/2)\cosh(r/2)}{e^r(1+(1+|x|)\sinh (r/2))} = 1
  \]
  and the assertion follows.
\end{proof}

\begin{figure}[!ht]
  \begin{center}
    \includegraphics[height=50mm]{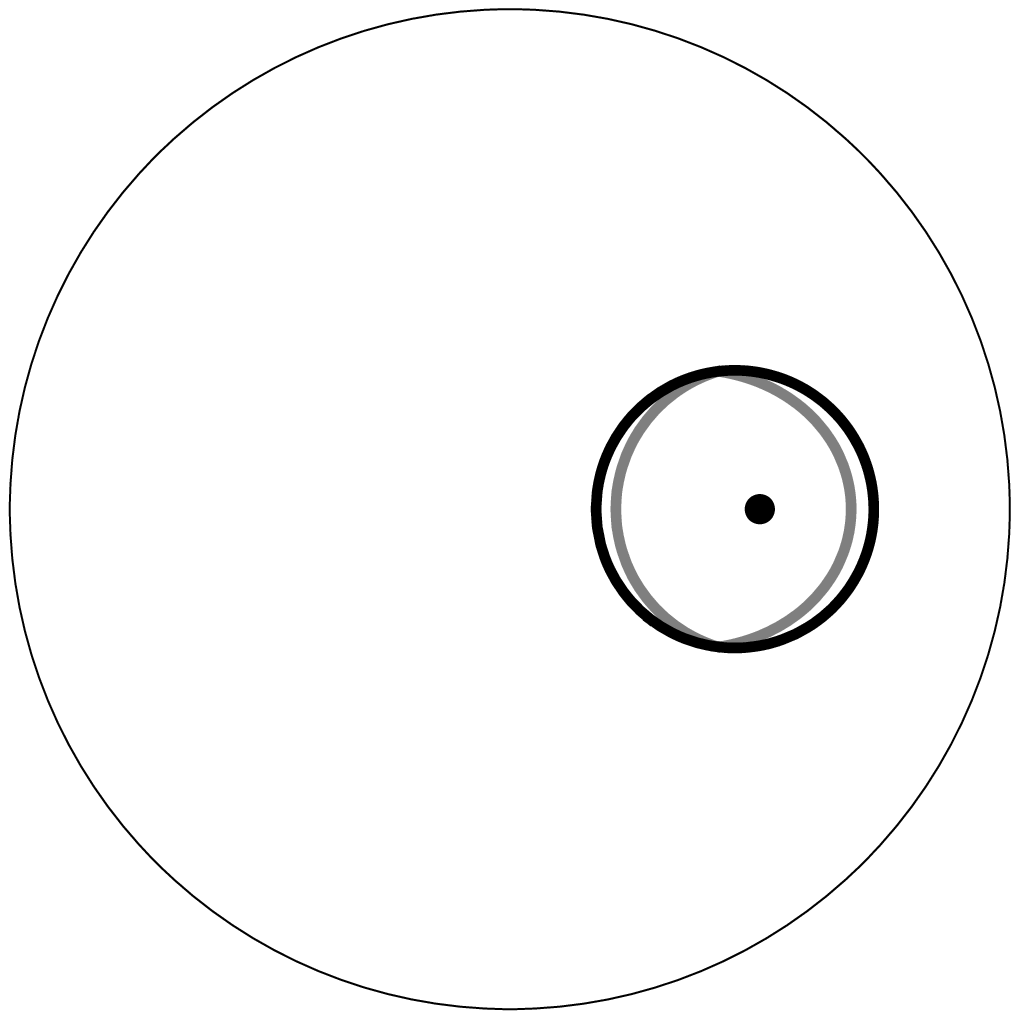}\hspace{1cm}
    \includegraphics[height=50mm]{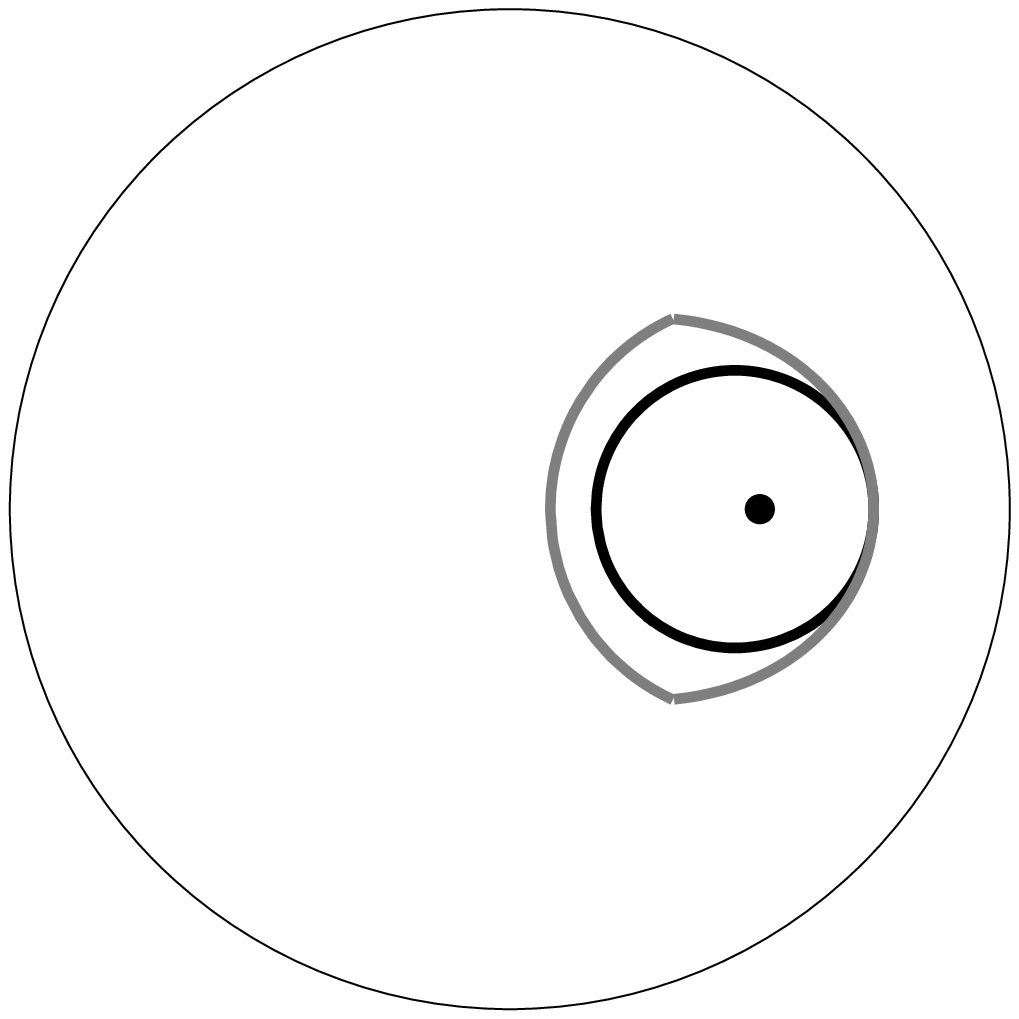}
    \caption{An example of inclusions of hyperbolic disks (black) and $j$-metric disks (gray) in the unit disk. The black dot is the center of the disks and the black thin circle is the unit circle.\label{fig1}}
  \end{center}
\end{figure}

\begin{corollary}\label{rhojrho}
  Let $G = \Bn$, $x \in G$ and $r > 0$. Then
  \[
    B_\rho(x,m) \subset B_j(x,r) \subset B_\rho(x,M),
  \]
  where
  \[
    m = \log \left( 1+\frac{2(e^r-1)}{1+|x|} \right)
  \]
  and
  \[
    M = \min \left\{ 2 \arcsinh \frac{e^r-1}{1+|x|} , \log\frac{2e^r-1-|x|}{1-|x|} \right\}.
  \]
  Moreover, the inclusions are sharp and $M/m \to 1$ as $r \to 0$.
\end{corollary}
\begin{proof}
  Assertion follows from Theorem \ref{jrhoj}.
\end{proof}

\begin{corollary}\label{jkj}
  Let $G = \Bn$, $x \in G$ and $r > 0$. Then
  \[
    B_j(x,m) \subset B_k(x,r) \subset B_j(x,M),
  \]
  where
\[
    m = \max \left\{ \log \left( 1+(1+|x|)\sinh \frac{r}{4} \right) , \log \left( 1+(1-|x|)\frac{e^{r/2}-1}{2} \right) \right\}
  \]
  and
  \[
    M = \log \left( 1+(1+|x|)\frac{e^r-1}{2} \right).
  \]
\end{corollary}\label{kjk}
\begin{proof}
  Assertion follows from Theorem \ref{jrhoj} and Proposition \ref{aputulos} (1).
\end{proof}

\begin{corollary}
  Let $G = \Bn$, $x \in G$ and $r > 0$. Then
  \[
    B_k(x,m) \subset B_j(x,r) \subset B_k(x,M),
  \]
  where
\[
    m = \log \left( 1+\frac{2(e^r-1)}{1+|x|} \right)
  \]
  and
  \[
    M = \min \left\{ 4 \arcsinh \frac{e^r-1}{1+|x|} , 2 \log \frac{2e^r-1-|x|}{1-|x|} \right\}.
  \]
\end{corollary}
\begin{proof}
  Assertion follows from Corollary \ref{jkj}.
\end{proof}

It is easy to verify that for $x \in \Bn$ we have $B_q(x,r) \subset \Bn$ if and only if $r < (1-|x|)/\sqrt{2(1+|x|^2)}$.

\begin{theorem}\label{jqj}
  Let $G = \Bn$, $x \in G$ and $r \in (0,r_0)$ for $r_0 = (1-|x|)/\sqrt{2(1+|x|^2)}$. We define real numbers $r_1 = |x|/\sqrt{1+|x|^2}$, $r_2=2|x|/(1+|x|^2)$ and intervals $I_1 = [0, \min\{ r_0,r_1 \})$, $I_2 = [ \min \{ r_0 ,r_1 \} , \min \{ r_0,r_2 \} )$ for $|x| < \sqrt{2}-1$, and $I_3 = [r_2,r_0)$ for $|x| < 2-\sqrt{3}$. Then
  \[
    B_j(x,m) \subset B_q(x,r) \subset B_j(x,M),
  \]
  where
  \[
    M = \log \frac{(1-|x|)(1-r^2(1+|x|^2))}{1-|x|-r(1+|x|^2)(r+\sqrt{1-r^2})}
  \]
  and
  \[
    m = \min \{ m_1,m_2 \}
  \]
  for
  \begin{eqnarray*}
    m_1 & = & \left\{ \begin{array}{ll} \displaystyle \log \left( 1+\frac{r(1+|x|^2)(\sqrt{1-r^2}-r |x|)}{(1-|x|) (1-r^2 (1+|x|^2))} \right), & \begin{array}{l}{\displaystyle r \in I_1 \textrm{ or}}\\{\displaystyle |x| < \sqrt{2}-1 \textrm{ and } r \in I_2,}\end{array}\\ \infty, & \textrm{otherwise},  \end{array} \right.\\
    m_2 & = & \left\{ \begin{array}{ll} \displaystyle \log \frac{(1+|x|)(1-r^2(1+|x|^2))}{1+|x|-r(r+\sqrt{1-r^2})(1+|x|^2)}, & |x| < 2-\sqrt{3} \textrm{ and } r \in I_3, \\ \infty, & \textrm{otherwise.} \end{array} \right.
  \end{eqnarray*}
  Moreover, the inclusions are sharp and $M/m \to 1$ as $r \to 0$.
\end{theorem}
\begin{proof}
  Because of symmetry of $G$ we may assume $x = t e_1$ for $t \in [0,1)$. Since $\partial \Ball{q}{x}{r}$ intersects the line $l = \{ z\in \Rn \colon z=e_1 s, s\in (-\infty,\infty) \}$ twice we denote $(\partial \Ball{q}{x}{r}) \cap l = \{ y_1,y_2 \}$. We assume that $y_1 \in (x,e_1)$ and $y_2 \in (x,-e_1)$.

  We prove first that $\Ball{q}{x}{r} \subset \Ball{j}{x}{M}$. Our idea is to show that
  \begin{equation}\label{Mjkj}
    \Ball{q}{x}{r} \subset B^n(x,|x-y_1|) \subset \Ball{j}{x}{M}.
  \end{equation}
  The first inclusion follows from Proposition \ref{aputulos} (4) and the observation that $|x| \le |x|/(1-r^2(1+|x|^2))$.

  The second inclusion of \eqref{Mjkj} follows from Lemma \ref{BjBn}, because $q(x,y_1) = r$ is equivalent to
  \[
    |y_1| = \frac{|x|+r\sqrt{1-r^2}(1+|x|^2)}{1-r^2(1+|x|)^2}
  \]
  and thus
  \[
    j(x,y_1) = \log \left( 1+\frac{|y_1|-|x|}{1-|y_1|} \right) = \log \frac{(1-|x|)(1-r^2(1+|x|^2))}{1-|x|-r(1+|x|^2)(r+\sqrt{1-r^2})} = M.
  \]

  We prove next that $\Ball{j}{x}{m} \subset \Ball{q}{x}{r}$. Our idea is to show that
  \begin{equation}\label{mjkj}
    \Ball{j}{x}{m} \subset B^n(x,|x-y_2|) \subset \Ball{q}{x}{r},
  \end{equation}
  where the second inclusion follows from Proposition \ref{aputulos} (4) and the observation that $|x| \le |x|/(1-r^2(1+|x|^2))$.

  The first inequality of \eqref{mjkj} follows from Lemma \ref{BjBn}, if $j(x,y_2) = m$. To show this we consider three cases: $y_2 \in [0,x)$, $y_2 \in (0,-x]$ and $y_2 \in (-x,-x/|x|)$.

  In the case $y_2 \in [0,x)$, $q(x,y_2) = r$ is equivalent to
  \[
    |y_2| = \frac{|x|+r\sqrt{1-r^2}(1+|x|^2)}{1-r^2(1+|x|^2)}
  \]
  and thus
  \[
    j(x,y_2) = \log \left( 1+\frac{|x|-|y_2|}{1-|x|} \right) = \log \left( 1+\frac{r(1+|x|^2)(\sqrt{1-r^2}-r |x|)}{(1-|x|) (1-r^2 (1+|x|^2))} \right) = m_1.
  \]

  In the case $y_2 \in [0,-x)$, $q(x,y_2) = r$ is equivalent to
  \begin{equation}\label{mcase2}
    |y_2| = \frac{|x|-r\sqrt{1-r^2}(1+|x|^2)}{1-r^2(1+|x|^2)}
  \end{equation}
  and thus
  \[
    j(x,y_2) = \log \left( 1+\frac{|x|+|y_2|}{1-|x|} \right) = \log \left( 1+ \frac{r(1+|x|^2)(\sqrt{1-r^2}-r |x|)}{(1-|x|)(1-r^2(1+|x|^2))} \right) = m_1.
  \]

  In the case $y_2 \in (-x,-x/|x|)$, $q(x,y_2) = r$ is equivalent to \eqref{mcase2} and thus
  \[
    j(x,y_2) = \log \left( 1+\frac{|x|+|y_2|}{1-|y_2|} \right) = \log \frac{(1+|x|)(1-r^2(1+|x|^2))}{1+|x|-r(r+\sqrt{1-r^2})(1+|x|^2)} = m_2.
  \]

 Sharpness of $m$ and $M$ follow from \eqref{Mjkj}, \eqref{mjkj} and the selection of $y_1$ and $y_2$.

  We finally show that $M/m \to 1$ as $r \to 0$. By the l'H\^opital's rule we obtain
  \[
    \lim_{r \to 0} \frac{M}{m} = \lim_{r \to 0} \frac{M}{m_1} = \lim_{r \to 0} \frac{(1+\alpha-\beta)(1-|x|+\gamma)}{(1-\alpha-\beta)(1-|x|-\gamma)} = 1,
  \]
  where $\alpha = 2r \sqrt{1-r^2}|x|$, $\beta = r^2(1-|x|^2)$ and $\gamma = r(\sqrt{1-r^2}-r)(1+|x|^2)$.
\end{proof}

\begin{figure}[!ht]
  \begin{center}
    \includegraphics[height=50mm]{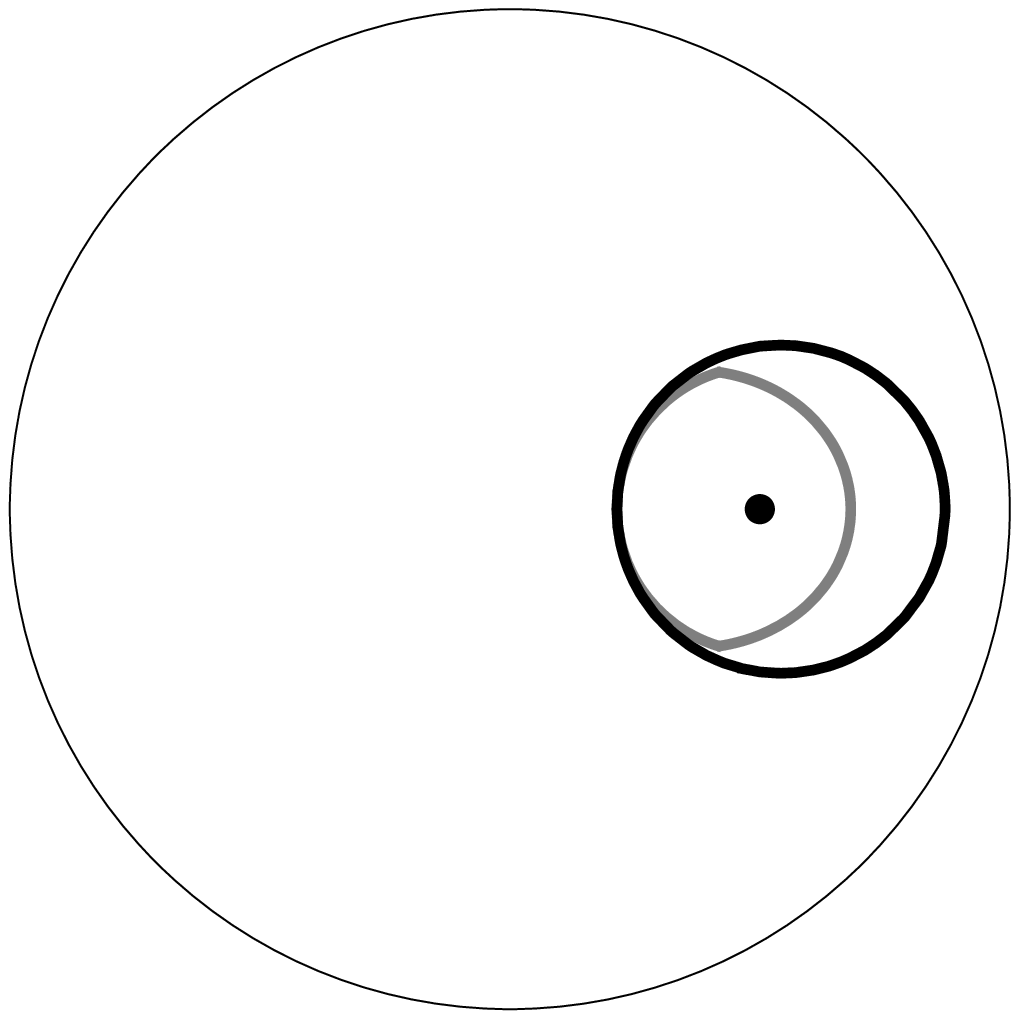}\hspace{1cm}
    \includegraphics[height=50mm]{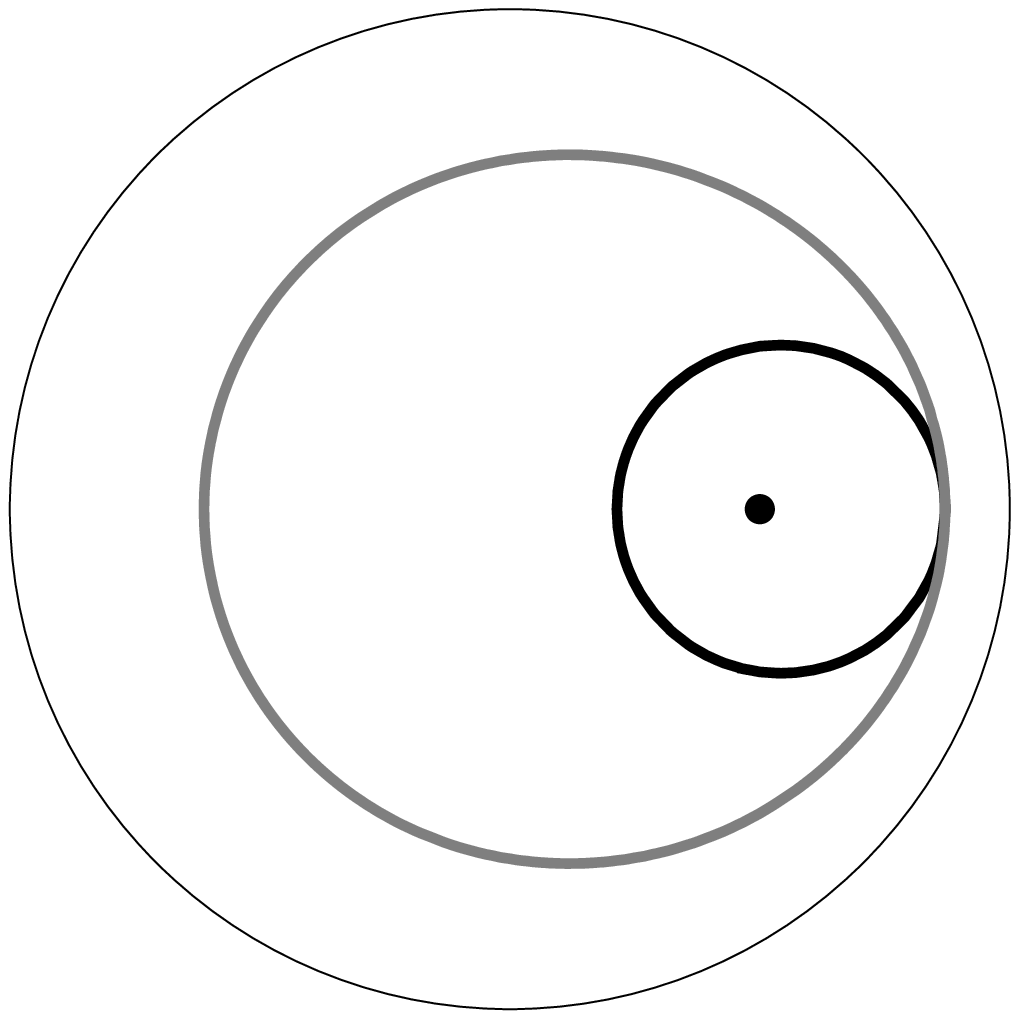}
    \caption{An example of inclusions of chordal disks (black) and $j$-metric disks (gray) in the unit disk. The black dot is the center of the disks and the black thin circle is the unit circle.\label{fig1}}
  \end{center}
\end{figure}

\begin{corollary}\label{qjq}
  Let $G = \Bn$, $x \in G$ and $r > 0$. Then
  \[
    B_q(x,m) \subset B_j(x,r) \subset B_q(x,M),
  \]
  where
  \[
    m = \frac{(1-e^{-r})(1-|x|)}{\sqrt{1+|x|^2}\sqrt{1+(e^{-r}(1-|x|)-1)^2}}
  \]
  and
  \begin{eqnarray*}
    M & = & \left\{ \begin{array}{ll} \displaystyle \frac{(e^r-1)(1-|x|)}{\sqrt{1+|x|^2}\sqrt{1+(e^r(1-|x|)-1)^2}}, & \displaystyle r \le \log \frac{1+|x|}{1-|x|},\\ \displaystyle \frac{(e^r-1)(1+|x|)}{e^r \sqrt{1+|x|^2}\sqrt{1+(e^{-r}(1-|x|)-1)^2}}, & \displaystyle r > \log \frac{1+|x|}{1-|x|}.\\ \end{array} \right.\\
  \end{eqnarray*}

  Moreover, the inclusions are sharp and $M/m \to 1$ as $r \to 0$.
\end{corollary}

\begin{remark}
  In Corollary \ref{qjq}, we have $\Ball{q}{x}{M} \subset \B^n$ if $M \le (1-|x|)/\sqrt{2(1+|x|^2)}$, which is equivalent to
  \[
    r \le \log \frac{2(1+|x|)}{1+2|x|-|x|^2}.
  \]
\end{remark}

\begin{proof}[Proof of Theorem \ref{mainthm1}]
  The radii $m_1$ and $m_2$ follow from Theorem \ref{jrhoj} and Corollary \ref{jkj}.

  The radius $m_3$ follows from Theorem \ref{jqj}.
\end{proof}

\begin{theorem}\label{rhoqrho}
  Let $G = \Bn$, $x \in G$ and $r \in (0,r_0)$ for $r_0 = (1-|x|)/\sqrt{2(1+|x|^2)}$. Then
  \[
    B_\rho(x,m) \subset B_q(x,r) \subset B_\rho(x,M),
  \]
  where
  \[
    m = 2 \arcsinh \frac{r(\sqrt{1-r^2}-r|x|)(1+|x|^2)}{\sqrt{1-|x|^2}a\sqrt{1-a^{-2}(|x|-r\sqrt{1-r^2}(1+|x|^2))^2}}
  \]
   and
  \[
    M = 2 \arcsinh \frac{r(\sqrt{1-r^2}+r|x|)(1+|x|^2)}{\sqrt{1-|x|^2}a\sqrt{1-a^{-2}(|x|+r\sqrt{1-r^2}(1+|x|^2))^2}}
  \]
  for $a=1-r^2(1+|x|^2)$.

  Moreover, the inclusions are sharp and $M/m \to 1$ as $r \to 0$.
\end{theorem}
\begin{proof}
  We prove the first inclusion $B_\rho(x,m) \subset B_q(x,r)$. Let $y \in \partial \Ball{\rho}{x}{m}$ with $|y| \le |z|$ for all $z \in \partial \Ball{\rho}{x}{m}$. By Lemma \ref{aputulos} (3) and (4)
  \[
    \Ball{\rho}{x}{m} \subset B^n(x,|x-y|) \subset \Ball{q}{x}{r}.
  \]
  Since $q(x,y) = r$ is equivalent to
  \[
    y = \frac{x}{|x|} \frac{|x|-r \sqrt{1-r^2}(1+|x|^2)}{1-r^2(1+|x|^2)}
  \]
  we obtain $\rho(x,y) = m$. The radius $m$ is sharp by the selection of $y$.

  We prove next the inclusion $B_q(x,r) \subset B_\rho(x,M)$. Let $y \in \partial \Ball{q}{x}{r}$ with $|y| \ge |z|$ for all $z \in \partial \Ball{q}{x}{r}$. By Lemma \ref{aputulos} (3) and (4)
  \[
    \Ball{q}{x}{r} \subset B^n(x,|x-y|) \subset \Ball{\rho}{x}{M}.
  \]
  Since $q(x,y) = r$ is equivalent to
  \[
    |y| = \frac{|x|+r \sqrt{1-r^2}(1+|x|^2)}{1-r^2(1+|x|^2)}
  \]
  we obtain $\rho(x,y) = M$. The radius $M$ is sharp by the selection of $y$.

  Clearly $M/m \to 1$ as $r \to 0$ and the assertion follows.
\end{proof}

\begin{figure}[!ht]
  \begin{center}
    \includegraphics[height=50mm]{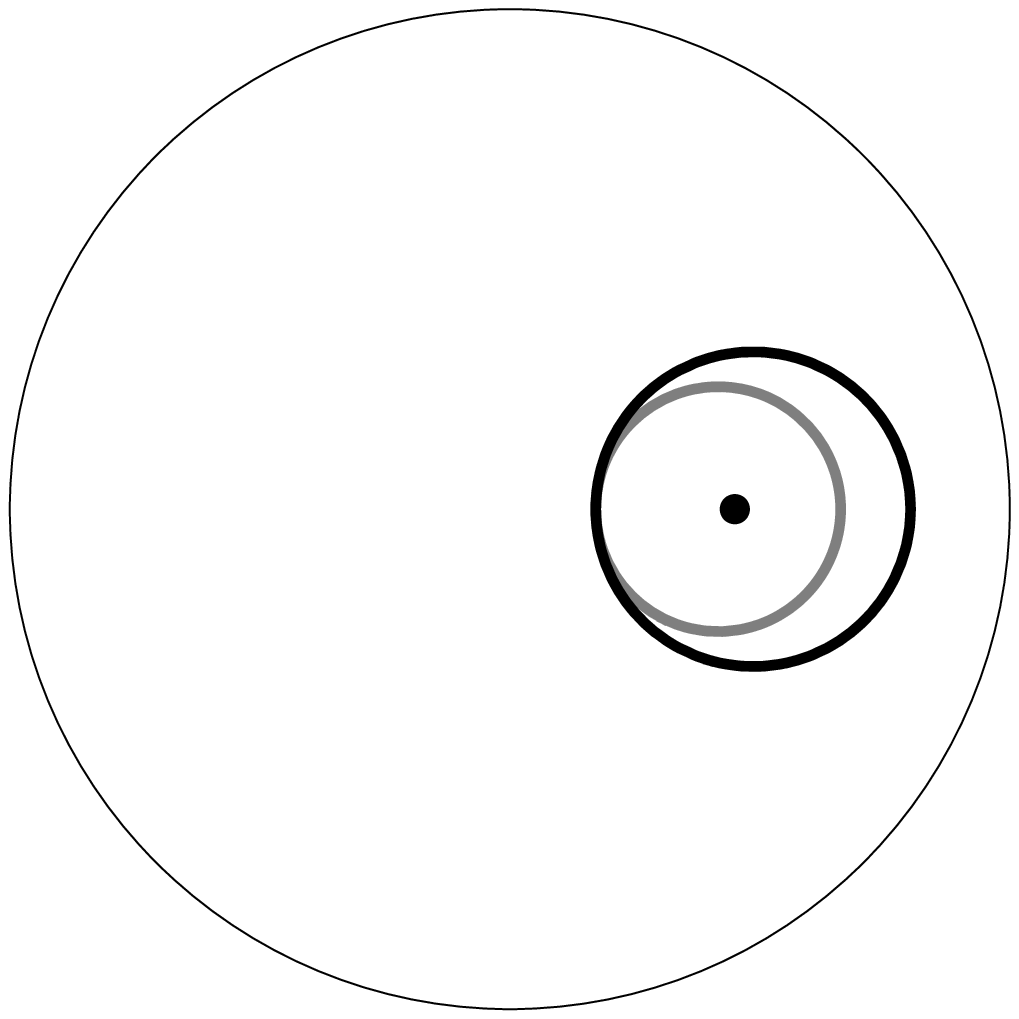}\hspace{1cm}
    \includegraphics[height=50mm]{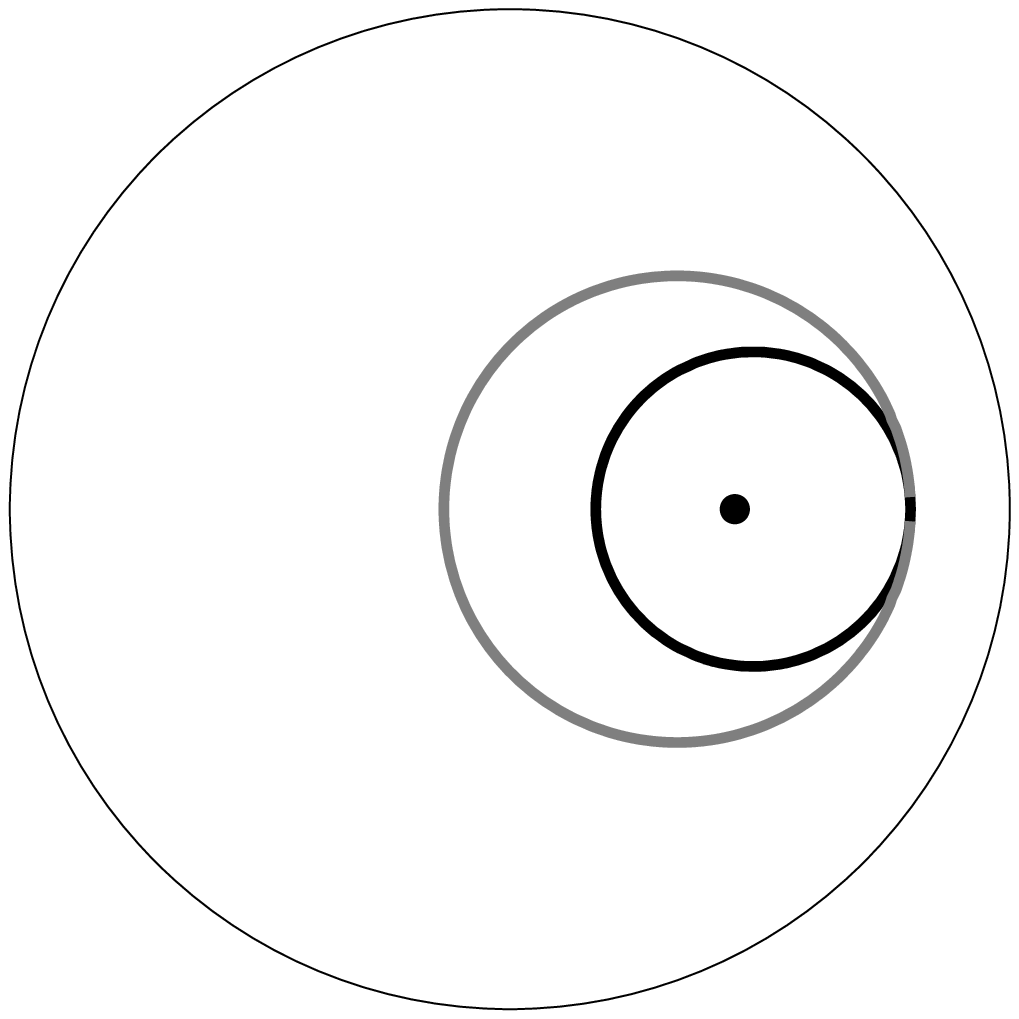}
    \caption{An example of inclusions of hyperbolic disks (gray) and chordal disks (black) in the unit disk. The black dot is the center of the disks and the black thin circle is the unit circle.\label{fig1}}
  \end{center}
\end{figure}

\begin{corollary}\label{kqk}
  Let $G = \Bn$, $x \in G$ and $r \in (0,r_0)$ for $r_0 = (1-|x|)/\sqrt{2(1+|x|^2)}$. Then
  \[
    B_k(x,m) \subset B_q(x,r) \subset B_k(x,M),
  \]
  where
  \[
    m = 2 \arcsinh \frac{r(\sqrt{1-r^2}-r|x|)(1+|x|^2)}{\sqrt{1-|x|^2}a\sqrt{1-a^{-2}(|x|-r\sqrt{1-r^2}(1+|x|^2))^2}}
  \]
   and
  \[
    M = 4 \arcsinh \frac{r(\sqrt{1-r^2}+r|x|)(1+|x|^2)}{\sqrt{1-|x|^2}a\sqrt{1-a^{-2}(|x|-r\sqrt{1-r^2}(1+|x|^2))^2}}
  \]
  for $a=1-r^2(1+|x|^2)$.
\end{corollary}

\begin{theorem}\label{qrhoq}
  Let $G = \Bn$, $x \in G$ and $r > 0$. Then
  \[
    B_q(x,m) \subset B_\rho(x,r) \subset B_q(x,M),
  \]
  where
  \[
    m = \frac{(1-|x|^2)\sinh \frac{r}{2}}{\sqrt{1+|x|^2}\sqrt{(1+|x|^2)\cosh r+2|x|\sinh r}}
  \]
  and
  \[
    M = \frac{(1-|x|^2)\sinh \frac{r}{2}}{\sqrt{1+|x|^2}\sqrt{(1+|x|^2)\cosh r-2|x|\sinh r}}.
  \]

  Moreover, the inclusions are sharp and $M/m \to 1$ as $r \to 0$.
\end{theorem}
\begin{proof}
  We prove first the inclusion $B_q(x,m) \subset B_\rho(x,r)$. Let $y \in \partial \Ball{q}{x}{m}$ with $|y| \ge |z|$ for all $z \in \partial \Ball{q}{x}{m}$. By Lemma \ref{aputulos} (3) and (4)
  \[
    \Ball{q}{x}{m} \subset B^n(x,|x-y|) \subset \Ball{\rho}{x}{r}.
  \]
  Since $q(x,y) = r$ is equivalent to
  \[
    |y| = \frac{2|x|+(1-|x|^2)\sinh r}{1+|x|^2+(1-|x|^2)\cosh r}
  \]
  we obtain $q(x,y) = m$. The radius $m$ is sharp by the selection of $y$.

  We prove next the inclusion $B_\rho(x,r) \subset B_q(x,M)$. Let $y \in \partial \Ball{\rho}{x}{r}$ with $|y| \ge |z|$ for all $z \in \partial \Ball{\rho}{x}{r}$. By Lemma \ref{aputulos} (3) and (4)
  \[
    \Ball{\rho}{x}{r} \subset B^n(x,|x-y|) \subset \Ball{q}{x}{M}.
  \]
  Since $\rho(x,y) = r$ is equivalent to
  \[
    y = \frac{x}{|x|} \frac{2|x|-(1-|x|^2)\sinh r}{1+|x|^2+(1-|x|^2)\cosh r}
  \]
  we obtain $q(x,y) = M$. The radius $M$ is sharp by the selection of $y$.

  Clearly $M/m \to 1$ as $r \to 0$ and the assertion follows.
\end{proof}

\begin{corollary}\label{qkq}
  Let $G = \Bn$, $x \in G$ and $r > 0$. Then
  \[
    B_q(x,m) \subset B_k(x,r) \subset B_q(x,M),
  \]
  where
  \[
    m = \frac{(1-|x|^2)\sinh \frac{r}{4}}{\sqrt{1+|x|^2}\sqrt{(1+|x|^2)\cosh \frac{r}{2}+2|x|\sinh \frac{r}{2}}}
  \]
  and
  \[
    M = \frac{(1-|x|^2)\sinh \frac{r}{2}}{\sqrt{1+|x|^2}\sqrt{(1+|x|^2)\cosh r-2|x|\sinh r}}.
  \]
\end{corollary}

Note that in Theorem \ref{qrhoq} and Corollary \ref{qkq}, $\Ball{q}{x}{M} \subset \B^n$ if $M \le (1-|x|)/\sqrt{2(1+|x|^2)}$.


\end{document}